\documentclass[11pt]{article}

\makeatletter
\renewcommand*\@fnsymbol[1]{\the#1}
\makeatother

\usepackage{amsmath}
\usepackage{amssymb}
\usepackage{amsfonts}
\usepackage{amsthm}
\usepackage[english]{babel}
\usepackage{latexsym}
\usepackage[hmargin=2.0cm,vmargin=2.3cm]{geometry}
\usepackage{enumerate}
\usepackage{nicefrac}
\usepackage{mathrsfs}
\usepackage[affil-it]{authblk}
\usepackage{datetime}
\usepackage{comment}
\usepackage{color}
\usepackage{graphicx}
\usepackage{epstopdf}

\theoremstyle{plain}
\newtheorem{theorem}{Theorem}[section]
\newtheorem{lemma}[theorem]{Lemma}
\newtheorem{proposition}[theorem]{Proposition}

\theoremstyle{definition}

\newtheorem{remark}[theorem]{Remark}

\newcommand{\cl}{\mathop{\rm cl}\nolimits}

\newcommand{\e}{\varepsilon}

\newcommand{\cF}{\mathcal F}

\newcommand{\cA}{\mathcal A}

\newcommand{\cP}{\mathcal P}

\newcommand{\probp}{\mathbb P}

\newcommand{\N}{\mathbb N}

\begin{document}

\title{
A simple characterization of tightness for\\
convex solid sets of positive random variables
}

\author{
\sc{Pablo Koch-Medina}\thanks{
Email: \texttt{pablo.koch@bf.uzh.ch}}
\,,
\sc{Cosimo Munari}\,\thanks{
Email: \texttt{cosimo.munari@bf.uzh.ch}}
\,,
\sc{Mario \v{S}iki\'c}\,\thanks{
Email: \texttt{mario.sikic@bf.uzh.ch}}
}
\affil{Center for Finance and Insurance, University of Zurich, Switzerland}

\date{April 2nd, 2017}

\maketitle

\abstract{
\noindent We show that for a convex solid set of positive random variables to be tight, or equivalently bounded in probability, it is necessary and sufficient that it is radially bounded, i.e.~that every ray passing through one of its elements eventually leaves the set. The result is motivated by problems arising in mathematical finance.
}

\bigskip

\noindent \textbf{Keywords}: random variables, convexity, solidity, tightness, radial boundedness.

\medskip

\noindent {\bf Mathematics Subject Classification}: 46A16, 46E30, 60A10


\parindent 0em \noindent

\section{Introduction}

In mathematical finance one has often to deal with sets of positive random variables that are {\em convex} and {\em solid}. The prototypical situations arise in the context of arbitrage pricing theory, see Delbaen and Schachermayer (1994), and in the context of utility maximization, see Kramkov and Schachermayer (1999), where the sets of interest consist of financial payoffs that are dominated, or superreplicated, by a given payoff. An important result in this direction was the bipolar theorem for convex solid sets of positive random variables established by Brannath and Schachermayer (1999) and later extended by Kupper and Svindland (2011) beyond the solid setting, which showed how to employ duality methods in spite of the lack of local convexity of the space of random variables. The combination of convexity and solidity has also been instrumental for the study of a variety of properties of sets of positive random variables inspired by problems from financial mathematics, namely numeraires, see Kardaras (2012) and Kardaras (2015), forward-convex convergence, see Kardaras and \v{Z}itkovi\'c (2013), and uniform integrability and local convexity, see Kardaras (2014).

\medskip

A central role in the previous contributions was played by the property of {\em tightness}, also called {\em boundedness in probability}. For convex sets of positive random variables that are closed (but not necessarily solid) this property is equivalent to {\em convex compactness}, see \v{Z}itkovi\'c (2009), which helps explain its ubiquity in duality and optimization problems involving positive random variables.

\medskip

The aim of this note is to provide a simple characterization of tightness for convex solid sets of positive random variables in terms of the property of {\em radial boundedness}, which amounts to requiring that any ray passing through some element of the given set eventually leaves the set. The result will automatically yield a characterization of convex compactness for convex solid sets of positive random variables that are closed. This characterization, which is new to the best of our knowledge, is appealing because radial boundedness is a property that is far easier to check than tightness or convex compactness.


\section{Notation and terminology}

We briefly recall some basic notions about sets of random variables. For more details we refer to the monograph by Aliprantis and Border (2006).

\medskip

Let $(\Omega,\cF,\probp)$ be a probability space. We denote by $L^0$ the set of all real-valued measurable functions on $\Omega$ modulo $\probp$-almost-sure equality. As usual, we will not distinguish between an equivalence class and any of its representatives. Each element of $L^0$ will be called a random variable.

\medskip

{\bf Vector structure}. The set $L^0$ is equipped with the canonical real vector space structure. A set $\cA\subset L^0$ is said to be {\em convex} if it contains the segment connecting any of its points, i.e.
\[
X,Y\in\cA, \ \lambda\in[0,1] \ \implies \ \lambda X+(1-\lambda)Y\in\cA.
\]

\smallskip

We say that $\cA$ is {\em radially bounded} whenever any ray passing through some of its points eventually leaves the set, i.e.
\[
X\in\cA\setminus\{0\} \ \implies \ \exists \lambda_X\in(0,\infty) \,:\, \lambda X\notin\cA, \ \forall \lambda\in(\lambda_X,\infty).
\]

\medskip

{\bf Order structure}. We equip the space $L^0$ with the lattice order defined by
\[
X\geq Y \ \iff \ \probp(X\geq Y)=1.
\]

\smallskip

We say that $X\in L^0$ is {\em positive} if $X\geq0$. The set of positive random variables is denoted by $L^0_+$. We say that $\cA\subset L^0$ is {\em solid} whenever
\[
X\in\cA, \ |Y|\leq|X| \ \implies \ Y\in\cA.
\]

\medskip

{\bf Topological structure}. We equip the space $L^0$ with the (metric) topology of convergence in probability given by
\[
X_n\to X \ \iff \ \probp(|X_n-X|>\e)\to0, \ \forall \e>0.
\]

\smallskip

The corresponding topology is induced, for instance, by the metric
\[
d(X,Y) = \int_\Omega\min(|X-Y|,1)d\probp.
\]
Recall that every convergent sequence in $L^0$ admits a subsequence that converges almost surely (to the same limit).

A set $\cA\subset L^0$ is said to be {\em tight} or {\em bounded (in probability)} if it is topologically bounded (note that every subset of $L^0$ is metrically bounded). This is equivalent to
\[
\forall \e>0 \ \exists M>0 \,:\, \sup_{X\in\cA}\probp(|X|>M)<\e.
\]

\smallskip

We say that $\cA$ is {\em closed (in probability)} if it contains the limit of any convergent sequence of its elements, i.e.
\[
(X_n)\subset\cA, \ X_n\to X \ \implies \ X\in\cA.
\]
The smallest closed set containing $\cA$ is called the {\em closure} of $\cA$ and is denoted by $\cl(\cA)$. Note that the closure of a solid, respectively tight, set is still solid, respectively tight.


\section{The main result}

Before we prove our announced characterization of tightness it is useful to state the following preliminary results. We start by showing that, when checking radial boundedness for a solid subset of $L^0_+$, we can restrict our attention to indicator functions. Here, for any $E\in\cF$ we denote by $1_E$ the element of $L^0$ that is equal to one on $E$ and is null on $\Omega\setminus E$.

\begin{lemma}
\label{lemma: rb for monotone}
A solid set $\cA\subset L^0_+$ is radially bounded if and only if
\begin{equation}
\label{radial boundedness property}
E\in\cF, \ \probp(E)>0 \ \implies \ \exists \lambda_E\in(0,\infty) \,:\, \lambda_E1_E\notin\cA.
\end{equation}
\end{lemma}
\begin{proof}
The ``only if'' implication is obvious. To prove the converse implication, assume that~\eqref{radial boundedness property} holds and take $X\in\cA\setminus\{0\}$ so that $\probp(X>\e)>0$ for a suitable $\e>0$. Now, set $E=\{X>\e\}\in\cF$ and take $\lambda_X=\lambda_E/\e$, where $\lambda_E\in(0,\infty)$ is as in~\eqref{radial boundedness property}. Then, for every $\lambda\in(\lambda_X,\infty)$ we have
\[
\lambda X \geq \lambda_X X = \frac{\lambda_E}{\e}X \geq \lambda_E1_E \notin\cA,
\]
so that $\lambda X\notin\cA$ by solidity. This proves that $\cA$ is radially bounded.
\end{proof}

\medskip

The second preliminary result shows that, while the closure of a radially bounded set need not be radially bounded, this is always true in the presence of solidity.

\begin{lemma}
\label{lemma: closed radially bounded}
Assume $\cA\subset L^0_+$ is convex and solid. If $\cA$ is radially bounded, then $\cl(\cA)$ is also radially bounded.
\end{lemma}
\begin{proof}
Assume $\cl(\cA)$ is not radially bounded. Since $\cl(\cA)$ is solid, we find a suitable $E\in\cF$ such that $\probp(E)>0$ and $\lambda1_E\in\cl(\cA)$ for every $\lambda>0$ by the preceding lemma. Then, for any $m\in\N$ there exists a sequence $(X^m_n)\subset\cA$ converging almost surely to $m1_E$. Note that $(1_EX^m_n)$ also consists of elements of $\cA$ by solidity and converges almost surely to $m1_A$ as well. Hence, by Egorov, there exists a set $E_m\in\cF$ with
\[
\probp(\Omega\setminus E_m) < \frac{\probp(E)}{2}\frac{1}{2^m}
\]
and such that $(1_EX^m_n)$ converges uniformly to $m1_A$ on $E_m$. In particular, we can always find some $n\in\N$ such that
\[
1_EX^m_n \geq 1_{E\cap E_m}X^m_n \geq \bigg(m-\frac{1}{m+1}\bigg)1_{E\cap E_m},
\]
so that $(m-\frac{1}{m+1})1_{E\cap E_m}\in\cA$ by solidity. Now, we claim that the measurable set
\[
F = E\cap\bigcap_{m\in\N}E_m \in \cF
\]
satisfies $\probp(F)>0$. Indeed, since
\[
E\setminus F = \bigcup_{m\in\N}E\cap(\Omega\setminus E_m),
\]
we obtain
\[
\probp(F) \geq \probp(E)-\sum_{m\in\N}\probp(\Omega\setminus E_m) \geq \probp(E)-\frac{\probp(E)}{2} > 0\,.
\]
Finally, since $F\subset E$, it follows that $(m-\frac{1}{m+1})1_F\in\cA$ for all $m\in\N$ by solidity. However, this in contrast to $\cA$ being radially bounded, hence we must conclude that $\cl(\cA)$ has to be radially bounded itself.
\end{proof}

\medskip

\begin{remark}
\label{simple example}
The assumption of solidity in the above lemma cannot be dropped. To see this, let $\Omega=\{\omega_1,\omega_2\}$ and assume $\cF$ coincides with the power set of $\Omega$ and $\probp$ is given by $\probp(\omega_1)=p$ and $\probp(\omega_2)=1-p$ for some $p\in(0,1)$. Moreover, consider the set
\[
\cA = \{X\in L^0 \,; \ X(\omega_1)\in(0,1]\}.
\]
Clearly, $\cA$ is convex and radially bounded but not solid. However,
\[
\cl(\cA) = \{X\in L^0 \,; \ X(\omega_1)\in[0,1]\}
\]
is not radially bounded.\hfill$\qed$
\end{remark}

\medskip

We are now in a position to prove our main result, which establishes the equivalence of tightness and radial boundedness for convex solid subsets of $L^0_+$.

\begin{theorem}
\label{proposition: radial boundedness}
Assume $\cA\subset L^0_+$ is convex and solid. Then, $\cA$ is tight if and only it is radially bounded.
\end{theorem}
\begin{proof}
To prove the ``only if'' implication, assume that $\cA$ is not radially bounded so that we find $X\in\cA\setminus\{0\}$ and an increasing divergent sequence $(\lambda_n)\subset(0,\infty)$ with $\lambda_n X\in\cA$ for all $n\in\N$. Since $X$ is nonzero, there exists $\delta\in(0,\infty)$ for which $E=\{X>\delta\}\in\cF$ satisfies $\probp(E)>0$. Now, take an arbitrary $M\in(0,\infty)$ and choose $n\in\N$ such that $\lambda_n>\frac{M}{\delta}$. Then, we easily see that
\[
\probp(\lambda_n X>M) \geq \probp(X>\delta).
\]
This shows that $\cA$ is not tight.

\smallskip

To establish the ``if'' implication, assume that $\cA$ is radially bounded but not tight. In view of Lemma~\ref{lemma: closed radially bounded}, we may assume without loss of generality that $\cA$ is closed. Since $\cA$ is not tight, there exist $\e>0$ and a sequence $(X_n)\subset\cA$ such that $E_n=\{X_n>2^n\}\in\cF$ satisfies $\probp(E_n)>\e$ for all $n\in\N$. Note that $X_n\geq2^n1_{E_n}$ for every $n\in\N$ and, hence, the sequence $(2^n1_{E_n})$ is also contained in $\cA$ by solidity. Now, consider the set
\[
E = \bigcap_{m\in\N}\bigcup_{n\geq m}E_n \in \cF.
\]
We claim that
\begin{equation}
\label{eq: main result auxiliary}
(m1_E) \subset \cA,
\end{equation}
which together with
\[
\probp(E) \ge \limsup_{n\rightarrow\infty}\probp(E_n) > \e
\]
will be in conflict with property~\eqref{radial boundedness property} and will therefore imply that the set $\cA$ must be tight.

\smallskip

To that effect, fix an arbitrary $m\in\N$ and choose $n(m)\in\N$ such that $2^{n(m)}>m$. Note that for any $N>n(m)$ we have
\[
\sum_{k=n(m)+1}^N\frac{1}{2^{k-n(m)}}2^k1_{E_k} =
2^{n(m)}\sum_{k=n(m)+1}^N1_{E_k} \ge
m1_{\bigcup_{k=n(m)+1}^NE_k}.
\]
Since $\cA$ contains $0$ by solidity, the left-hand side above belongs to $\cA$ for every $N>n(m)$ by convexity. As a result, the solidity of $\cA$ implies that the right-hand side also belongs to $\cA$ for every $N>n(m)$. Noting that
\[
m1_{\bigcup_{k=n(m)+1}^NE_k} \to m1_{\bigcup_{k\geq n(m)+1}E_k}
\]
as $N\to\infty$, we infer from the closedness of $\cA$ that the above limit is also an element of $\cA$. Since $E\subset\bigcup_{k\geq n(m)+1}E_k$, we can use solidity once more to obtain $m1_E\in\cA$. This, in view of the fact that $m$ was arbitrary, yields~\eqref{eq: main result auxiliary} and concludes the proof.
\end{proof}

\medskip

\begin{remark}
It is clear from the proof of Theorem~\ref{proposition: radial boundedness} that tightness always implies radial boundedness, regardless of whether the set is convex or solid. Hence, the ``only if'' implication in Theorem~\ref{proposition: radial boundedness} remains true for an arbitrary subset of $L^0$.\hfill$\qed$
\end{remark}

\medskip

\begin{remark}
The ``if'' implication in Theorem~\ref{proposition: radial boundedness} is not true without assuming convexity and solidity.

\smallskip

{\em Solidity}. To show that solidity is necessary, note that Remark~\ref{simple example} provides an example of a non-tight subset of $L^0_+$ that is convex and radially bounded but not solid. The example can be trivially extended to any probability space.

\smallskip

{\em Convexity}. To show that convexity is also necessary, assume that $\Omega$ contains no atom, i.e.~no set $A\in\cF$ with $\probp(A)>0$ and such that any $E\in\cF$ that is contained in $A$ satisfies either $\probp(E)=0$ or $\probp(E)=\probp(A)$. In this case, we find random variables having any prescribed distribution; see F\"{o}llmer and Schied (2011). Now, consider a sequence $(X_n)\subset L^0_+$ of independent and identically distributed random variables with $\probp(X_1=0)=\probp(X_1=1)=\frac{1}{2}$ and define the set
\[
\cA = \{X\in L^0_+ \,; \ \exists n\in\N \,:\, nX_n\geq X\}.
\]
The set $\cA$ is clearly solid but not convex. Moreover, it is not tight since for any $M>0$ we have $\probp(nX_n\geq M)=\frac{1}{2}$ as soon as $n\geq M$. We now show that $\cA$ is radially bounded. To this end, assume by way of contradiction that $\cA$ is not radially bounded so that we find $E\in\cF$ satisfying $\probp(E)>0$ and $m1_E\in\cA$ for every $m\in\N$ by Lemma~\ref{lemma: rb for monotone}. By definition of $\cA$, for each $m\in\N$ there exists $n_m\in\N$ such that $n_m X_{n_m}\geq m1_E$. Since the elements of the sequence $(X_n)$ have Bernoulli distribution, we must have $E\subset\bigcap_{m\in\N}\{X_{n_m}=1\}$. Moreover, noting that $n_m\geq m$, we infer from independence that
\[
\probp\bigg(\bigcap_{m\in\N}\{X_{n_m}=1\}\bigg) = \prod_{m\in\N}\probp(X_{n_m}=1) = 0.
\]
This implies that $\probp(E)=0$, in contrast to our initial assumption on $E$. As a result, we conclude that $\cA$ must be radially bounded.\hfill$\qed$
\end{remark}

\medskip

In the context of a purely atomic probability space the convexity assumption in the preceding theorem can be omitted. Recall that $(\Omega,\cF,\probp)$ is said to be purely atomic if any set $E\in\cF$ with $\probp(E)>0$ contains an atom $A\in\cF$ satisfying $\probp(A)>0$. For more details about purely atomic spaces we refer to Bogachev (2007).

\begin{proposition}
\label{prop: atomic spaces}
Let $(\Omega,\cF,\probp)$ be purely atomic and assume $\cA\subset L^0_+$ is solid. Then, $\cA$ is tight if and only if it is radially bounded.
\end{proposition}
\begin{proof}
The ``only if'' implication can be established as above. To prove the ``if'' implication, assume that $\cA$ is radially bounded. If $\cA$ is not tight, then we find $\e\in(0,1)$ such that for any $n\in\N$ there exists a random variable $X_n\in\cA$ such that $\probp(X_n>n)\geq\e$. By assumption, there exists a measurable partition $\cP$ of $\Omega$ consisting of atoms. Since any probability space admits at most countably many atoms, $\cP$ contains at most countably many elements. Denote by $N\in\N$ some index for which we find atoms $A_1,\dots,A_N\in\cP$ satisfying
\[
\sum_{k=1}^N\probp(A_k)>1-\e.
\]
Now, fix $n\in\N$ and define $E_n=\{X_n>n\}\in\cF$. Then, we must find $k_{n}\in\{1,\dots,N\}$ such that $\probp(E_n\cap A_{k_n})>0$, for otherwise
\[
\e \leq \probp(E_n) \leq 1-\sum_{k=1}^N\probp(A_k) < \e.
\]
Without loss of generality we may assume that $k_n=k$ for all $n\in\N$, where $k$ is a given element of $\{1,\dots,N\}$. As a result, for any $n\in\N$ it follows that
\[
\probp(E_n\cap A_{k_n}) = \probp(E_n\cap A_k) = \probp(A_k)
\]
and therefore
\[
X_n \geq n1_{E_n\cap A_{k_n}} = n1_{A_k},
\]

\smallskip

where we used that $A_k$ is an atom. This yields $n1_{A_k}\in\cA$ for all $n\in\N$ by solidity, which is in contrast to radial boundedness.
\end{proof}


\end{document}